\documentclass[a4paper,12pt]{amsart}
\usepackage[dvips]{graphics}
\usepackage[matrix,arrow,tips,curve]{xy}
\usepackage[english]{babel}
\usepackage[leqno]{amsmath}
\usepackage{amssymb,amsthm}
\usepackage{amscd}
\usepackage{enumerate}

\input{xy}
\xyoption{all}

\newtheorem{thm}{Theorem}[section]
\newtheorem{lemma}[thm]{Lemma}

\newtheorem{prop}[thm]{Proposition}

\newtheorem{cor}[thm]{Corollary}

\newtheorem{summ}[thm]{Summary}

\newtheorem{conj}[]{Conjecture}
\theoremstyle{remark}
\newtheorem{remark}[thm]{Remark}

\theoremstyle{definition}
\newtheorem{defi}[thm]{Definition}

\newtheorem{example}[thm]{Example}

\newcommand{\la}{\longrightarrow}
\newcommand{\ha}{\hookrightarrow}

\newcommand{\ov}{\overline}

\newcommand{\Div}{\operatorname{Div}}

\newcommand{\supp}{\operatorname{Supp}}

\newcommand{\Pic}{\operatorname{Pic}}
\newcommand{\Tw}{\operatorname{Tw}}
\newcommand{\Jac}{\operatorname{Jac}}

\newcommand{\Prin}{\operatorname{Prin}}

\newcommand{\mdeg}{\underline{\operatorname{deg}}}
\newcommand{\val}{\operatorname{val}}
\newcommand{\ord}{\operatorname{ord}}

\def\L{\mathcal L}
\def\O{\mathcal O}

\def\X{\mathcal X}

\newcommand{\PP}{\mathbb{P}}
\newcommand{\Z}{\mathbb{Z}}

\newcommand{\N}{\mathbb{N}}

\newcommand{\G}{G}

\def\md{\underline{d}}
\def\mo{\underline{0}}

\def\me{\underline{e}}
\def\mk{\underline{k}}
\def\mt{\underline{t}}

\newcommand{\MaG}{M^{\rm alg}(G)}
\newcommand{\Mgb}{\ov{M_g}}

\newcommand{\ooG}{\overline{G}}

\newcommand{\oov}{\overline{v}}

\newcommand{\ooX}{\overline{X}}

\renewcommand{\div}{\mathrm{div}}
\newcommand{\rmax}{r^{\rm{max}}}
\newcommand{\RX}{\rmax(X,\md)}

\newcommand{\rdel}{r(X,\delta)}
\newcommand{\ralg}{r^{\rm{alg}}}
\newcommand{\rgdel}{\ralg(G,\delta)}

\newcommand{\lm}{G^{\bullet}}

\begin{document}
 \title{Rank of divisors on graphs: an algebro-geometric analysis}
 \author{ Lucia Caporaso}

\address{Dipartimento di Matematica,
Universit\`a Roma Tre,
Largo San Leonardo Murialdo 1,
00146 Roma (Italy)}
\email{caporaso@mat.uniroma3.it}
%\keywords{}
%\subjclass[2000]{14TXX, 05CXX}
%
 \maketitle
\begin{center}

 %\today

\

{\it  Dedicated to Joe Harris, for his sixtieth birthday.}
\end{center}

 \tableofcontents

 \

The goal of this paper is to apply the divisor theory for graphs to  the theory of linear series on singular algebraic curves, and to propose an algebro-geometric interpretation for the rank of divisors on graphs. 
Let us begin with a simple question.

What is the maximum dimension of a linear series of degree
$d\geq 0$ on a smooth projective  curve of genus $g$?

We know what the answer  is. If $d\geq 2g-1$ by Riemann's theorem  every complete linear series of degree $d$ on every smooth curve of genus $g$ has dimension $d-g$. If $d\leq 2g-2$ the situation is more interesting: Clifford's theorem states that the answer is $\lfloor d/2\rfloor$, and the bound is achieved only by certain linear series on hyperelliptic curves; see \cite{ACGH}.

 Now let us look at the combinatorial side of the problem. The  dual graph of any smooth curve of genus $g$ is the (weighted) graph with one vertex of weight equal to $g$ and no edges, let us denote it by $G_g$. This graph admits a unique divisor of degree $d$, whose rank, as we shall see,  is   equal to $d-g$ if  $d\geq 2g-1$, and to $\lfloor d/2\rfloor$ otherwise.

We draw the following conclusion: the maximum dimension of a linear series of degree $d$ on a smooth  curve of genus $g$ equals the rank of the degree $d$ divisor on the dual graph of  the curve.
 In symbols, denoting by $\md$ the unique divisor of degree $d$ on $G_g$ and by $r_{G_g}(\md)$ its rank
 (see below),
\begin{equation}
\label{toy}
 r_{G_g}(\md)=\max\{r(X,D),\  \  \forall X\in M_g,\  \forall D\in \Pic^d(X) \} 
\end{equation}
where $M_g$ is the moduli space  of smooth projective curves of genus $g$. 
 This is quite pleasing for at least two reasons. First, the graph is fixed, whereas the curve varies (in a  moduli space of dimension $3g-3$  if $g\geq 2$); also the divisor on $G_g$ is fixed, whereas $\Pic^d(X)$ has dimension $g$. Second: computing the rank of a divisor on a graph
is simpler than computing the dimension of a linear series on a curve; a computer can do that.

 Therefore we shall now ask how this phenomenon generalizes to singular curves.
For every graph $G$  we have
 a  family, $\MaG$, of curves  having dual graph equal to   $G$.
We want  to
 give an interpretation of the rank of a  divisor  on   $G$
 in terms of linear series on curves in $\MaG$.
 
 This is  quite a delicate issue, as  for such curves we do not have a good control on the dimension of a linear series; in fact, as we shall see, both  Riemann's theorem and   Clifford's theorem fail.
Furthermore,
 asking for the maximal dimension of a linear series of degree $d$   is not so interesting, as the answer easily turns out to be   $+\infty$.
By contrast, the rank of a divisor of degree $d\geq 0$ on a graph is always at most equal to $d$.
In fact, to set-up the problem    precisely we need   a few more details.
Let us assume some of them for now, and continue with this overview.

 For any curve $X$ having $G$ as dual graph, we have an identification of the set of irreducible components of $X$ with the set of vertices, $V(G)$, of $G$,
and  we write
\begin{equation}
\label{Vdec}
X=\cup_{v\in V(G)}C_v.
\end{equation}
 The group of divisors of   $G$ is   the free abelian group, $\Div G$,  generated by $V(G)$.
Hence there is a natural map sending a  Cartier divisor $D$ on $X$ to a divisor on $G$:
 $$
 \Div X \la \Div (G);\  \  \  D \mapsto \sum_{v\in V(G)}(\deg D_{|C_v}) v,
 $$
so that
the divisor of $G$ associated to $D$ is the multidegree of $D$;
 the above map descends to $\Pic (X)\to \Div (G)$, as linearly equivalent divisors  have the same 
 multidegree.
Therefore we can write
\begin{equation}
\label{picdec}
\Pic (X)=\bigsqcup_{\md\in \Div (G)} \Pic^{\md}(X).
\end{equation}
%That is, divisors of $G$ are in
%bijective correspondence with  connected components of $\Pic (X)$.
%(each of which is   a  $g$-dimensional algebraic variety).

On the other hand,  linearly equivalent divisors   on $G$ have the same rank,
so the combinatorial rank is really a function on 
divisor classes.
Let $\delta\in \Pic (G)$ be a divisor class on $G$ and write $r_G(\delta):=r_G(\md)$ for any representative $\md\in \delta$.

How does $r_G(\delta)$ relate to 
$r(X,L)$ as $X$ varies among curves having $G$ as dual graph, and $L\in \Pic(X)$  varies by keeping its   multidegree   class equal to $\delta$?
We conjecture that  the following identity holds:
\begin{equation}
\label{toy1}
r_G(\delta)= \max_{X\in \MaG}\Bigr\{\min_{\md\in\delta}\bigr\{ \  \  \max_{L\in \Pic^{\md}(X)}\{r(X,L)\}\bigl\}\Bigl\}.
\end{equation}
An accurate discussion of  this conjecture is at the beginning of Section~\ref{sec2}.
In  Section~\ref{sec1}, after some combinatorial preliminaries,  a comparative analysis
 of the graph-theoretic and algebraic situation  is carried out highlighting differences and analogies; this also serves as motivation.
In Section~\ref{sec2}  we prove the above identity   in a series of cases,  summarized  precisely at the end of the paper.

The techniques we use  are mostly algebro-geometric, while the combinatorial aspects
are kept at a minimum. The hope is, of course,  that using more sophisticated combinatorial arguments the validity range of  above identity
   could be completely  determined.
 
I am grateful to Margarida Melo and to the referee for some very useful remarks.

 \section{Combinatorial and algebraic rank}
 \label{sec1}

We apply the following conventions throughout the paper.

$X$ is a projective algebraic curve over some algebraically closed field.

$X$ is connected, reduced and has at most nodes as singularities. 

$G$ is a finite connected, vertex weighted graph. 

Capital letters $D,E,\ldots$ are  Cartier divisors on curves.

Underlined  lowercase letters  $\md,\me,\ldots$ are divisors on graphs.

$r(X,D):=h^0(X,D)-1$ is
the   (algebraic) rank of $D$    on   $X$.

$r_G(\md)$ is the  (combinatorial) rank of $\md$ on   $G$.

$\Div(*)$ is the set of divisors on $*$. $\Div_+(*)$  
  the set of effective  divisors.
  
  $\Div^d(*)$ is the set of divisors of degree $d$, for $d\in \Z$.

  $\sim$  \  is the  linear  equivalence on  $\Div^d(*)$.

$\Pic(*):=\Div(*)/\sim$  and $\Pic^d(*):=\Div^d(*)/\sim$.

\subsection{Basic divisor theory on graphs}
\label{bdef}
We begin by reviewing the combinatorial setting following
  \cite{BNRR}  and   \cite{AC}. The basic reference is \cite{BNRR}, which deals with loopless weightless graphs, we use the extension to general weighted graphs given  in \cite{AC}; see \cite{AB} for a different approach.

Let $G$ be a (finite, connected, weighted) graph; we allow loops.
  We write $V(G)$ and $E(G)$ for its vertex set  and edge set;
$G$ is given a   weight function $\omega:V(G)\to \Z_{\geq 0}$.
If $\omega =0$ we say that $G$ is {\it weightless}.
The genus of $G$ is $b_1(G)+\sum_{v\in V(G)}\omega(v)$.
  
We always fix an ordering   $V(G)=\{v_1,\ldots, v_{\gamma}\}$.
The group of divisors of $G$   is   the free abelian group on $V(G)$:
$$
 \Div (G) :=\{\sum _{i=1}^{\gamma}d_i v_i,\  d_i\in \Z \}\cong \Z^{\gamma}.
$$ 
Throughout the paper we identify $ \Div (G)$ with $ \Z^{\gamma}$,
so that divisors on graphs are usually represented by ordered sequences of integers, $\md = (d_1,\ldots, d_{\gamma})$; we write $\md \geq 0$   if $d_i\geq 0$ for every $i=1,\ldots, \gamma$.
 
 We set
$ 
|\md | =\sum_{i=1}^{\gamma}d_i,
$ 
so that    $\Div^d(G)=\{\md\in \Div (G):\  |\md|=d\}$; also $\Div_+(G):=\{\md\in \Div (G):\  \md \geq 0\}$.

For $v\in V(G)$ we denote   by $\md(v)$ the coefficient of $v$ in $\md$, so that  $\md(v_i)=d_i$. 

If $Z\subset V(G)$ we write $\md(Z)=\sum_{v\in Z}\md(v)$ and $\md_Z=(\md(v),\  \forall v\in Z)\in \Z^{|Z|}$. We set    $Z^c=V(G)\smallsetminus Z$.

The local geometry of $G$ can be described by its so-called  intersection product, which we are going to define.
Fix two vertices $v$ and $w$ of $G$; we want to think of   $v$ and $w$ as ``close" in $G$ if they are joined by some edges.
To start with we set,  if $v\neq w$,
$$
(v\cdot w):= 
\text{number of edges joining }v  \text{ and } w.$$
So, the greater $(v\cdot w)$ the closer $v$ and $w$.
Next we set  
\begin{equation}
\label{vv}
(v\cdot v)=-\sum_{w\neq v}(v\cdot w) 
\end{equation}
and the  {\it intersection} product, $\Div (G)\times \Div(G)\to \Z$, is defined as   the  
$\Z$-linear extension of
$(v,w)\mapsto (v\cdot w)$.

Given $Z,W\subset V(G)$, we shall frequently abuse notation by writing
$(W\cdot Z)=\sum_{w\in W, z\in Z}(w\cdot z)$.
 Notice that if $v\not\in W$ the quantity $(v\cdot W)$ is the number of edges joining $v$ with a vertex of $W$, whereas if $v\in W$ we have $(v\cdot W)\leq 0$

We are going to study functions on $G$, and their divisors.
A   rational function  $f$ on $G$ is a map
$ 
f:V(G)\to \Z.
$  
To define the associated divisor, $\div (f)$, we proceed   in analogy with   classical geometry.
We begin by requiring that if $f$ is constant  its divisor be equal to $0$.
The set of rational functions on $G$ is a group under addition;
so  
 we require that if $c:V(G)\to \Z$ is constant  then 
$ 
\div(f+c)=\div (f).
$ 
Now we need to study  the analogue of zeroes  and  poles, i.e. the local behaviour of a function near each $v\in V(G)$.
We write 
$$
\div (f):=\sum_{v\in V(G)}\ord_v(f) v
$$
where $\ord_v(f) \in \Z$ needs to be defined so as to  depend on the behaviour of $f$ near $v$, that is on the value of $f$ at each $w$ close to $v$, and on how close   $v$ and $w$ are.
We are also requiring  that
$\ord_v(f)$ 
  be invariant under 
 adding a constant 
 to $f$,
this suggests that $\ord_v(f)$ be a function of the difference $f(v)-f(w)$,
 proportional to $(v\cdot w)$. That was  an intuitive
motivation for the following definition 
\begin{equation}
\label{ord}
\ord_v(f):=\sum_{w\neq v}(f(v)-f(w))(v\cdot w).
\end{equation}
Loosely speaking,
  $\ord_v(f)=0$ means $f$ is  locally constant  at $v$, and
   $\ord_v(f)>0$
(resp.   $\ord_v(f)<0$),
means $v$ is a   local maximum  for $f$ (resp. a  local minimum).
 
Notice the following useful simple fact.
\begin{remark}
\label{cnerlm}
Let $Z\subset V(G)$ be the set of vertices where the function $f$ takes its minimum value. Then $\div(f)(Z)\leq - (Z\cdot Z^c)$ and for every $v\in Z$
we have $\div(f)(v)\leq 0$.
\end{remark}

Note that $\ord_v(f)=-\ord_{-f}(v)$ and $\ord_v(f)+ \ord_v(g)= \ord_{v}(f+g)$.
The divisors of the form $\div (f)$ are called {\it principal},
and are easily  seen to have degree zero. Thus they form a subgroup of $\Div^0(G)$,
denoted by $\Prin(G)$.

Two divisors  $\md,\md'\in \Div(G)$  are  linearly  equivalent,   written $\md \sim  \md'$,  if
$ 
\md-\md' \in \Prin (G).
$ 
We write $\Pic(G)=\Div(G)/\sim $;
we usually denote an element of $\Pic(G)$ by  $\delta$ and write $\md\in \delta$ for a representative; we also  write $\delta =[\md]$.
Now, $\md\sim \md'$ implies $|\md|=|\md'|$ hence we set
$$
\Pic^d(G)
=\Div^d(G)/\sim 
$$
(often in the graph-theory literature the notation $\Jac(G)$ is used for what we here denote by $\Pic(G)$  to stress the analogy with algebraic geometry).
 
 The group $\Pic^0(G)$ appears in several different places of the mathematical literature, with various names and notations; see for example \cite {BdlHN}, 
 \cite{OS}, \cite{raynaud}.
 
 It is   well known  that $\Pic^d(G)$  is a finite set whose cardinality 
 equals the complexity, i.e.    the number of spanning trees, of the graph $G$.

\begin{remark}
The intersection product does not depend on the   loops or the weights of $G$, hence the same holds for $\Prin(G)$ and $\Pic(G)$.
\end{remark}

To define the combinatorial rank we proceed in two steps, treating loopless, weightless graphs first.

Let $G$ be a loopless, weightless graph, and $\md\in \Div(G)$.
Following \cite{BNRR}, we  define the {\it (combinatorial)  rank} of $\md$   as follows 
 \begin{equation}\label{trasl}
r_G(\md)=\max\{k:  \forall  \me\in \Div^k_+(G)  \  \ \exists  \md'\sim  \md \text{ such that } \md'-\me\geq 0\}
\end{equation}
with $r_{G}(\md)=-1$ if the set on the right is empty. 

\

The combinatorial rank   defined in \eqref{trasl} satisfies a Riemann-Roch formula (see below)
if the graph  is free from loops and weights, but not in general. 
This is why a different definition is needed 
for weighted graphs admitting loops.
To do that 
we introduce the weightless, loopless graph $\lm$ 
 obtained from $G$ 
 by 
 first attaching $\omega(v)$ loops based at $v$ for every $v\in V(G)$, and then
 by
 inserting a vertex
in every loop edge. 
This graph  $\lm$  (obviously free from loops)   is assigned the zero
weight function.  Now  $G$ and $\lm$ have the same genus.

As $V(G)\subset V(\lm)$ we have a natural injection
$ 
\iota:\Div(G)\ha \Div (\lm).
$ 
%which identifies $\Div (G)$ as the subgroup of divisors on $\lm$ having  coefficient
%equal to zero on all the vertices in $V(\lm)\smallsetminus V(G)$.
It is easy to see that $\iota(\Prin(G))\subset \Prin(\lm)$, hence we have  
\begin{equation}
\label{inj.}
\Pic(G)\ha \Pic(\lm).
\end{equation}
We   define the rank for a divisor $\md$ on any graph $G$ as follows:
\begin{equation}
\label{rankwl}
r_G(\md ):=r_{\lm}(\iota(\md)) 
\end{equation}
where the right-hand-side is defined in \eqref{trasl}.

\begin{remark}
\label{lininv}
If $\md\sim \md'$ we have $r_G(\md)=r_G(\md')$.
\end{remark} 
\begin{example}
The picture below represents $G^{\bullet}$ for a graph having one vertex of weight 1 and one  loop based at a vertex of weight zero. 
We have $\Pic^0(G)=0$ and it is easy to check that $\Pic^0(G^{\bullet})\cong\Z/2\Z\oplus \Z/2\Z$.
Consider the divisor $v\in \Div(G)$; then 
$r_G(v)= 0$.
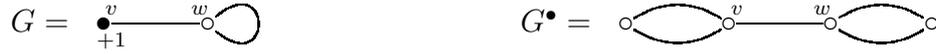
\begin{figure}[h]
\begin{equation*}
\xymatrix@=.5pc{
&&&&&&&&&&&&&&\\
G = &*{\bullet} \ar@{-}[rrr]_<{+1}^<{v}^>{w}&&
&*{\circ}\ar@{-}@(ur,dr)&&
&&&&&&&  
G^{\bullet} = &*{\circ}\ar@{-} @/^.6pc/[rrr]\ar@{-} @/_.6pc/[rrr]&&&*{\circ}\ar@{-}[rrr]^<{v}^>{w}&&
&*{\circ}\ar@{-} @/^.6pc/[rrr]\ar@{-} @/_.6pc/[rrr]&&&*{\circ}&&
&&&&\\
 &&&&&&&&&}
\end{equation*}
\caption{Weightless loopless model of $G$}
\end{figure}

In our  figures, weight-zero vertices are represented by a ``$\circ$".
\end{example}
It is clear that different graphs may have the same  $G^{\bullet}$, see for example the
picture  in  the proof of \ref{g2thm}.
Other examples will be given in the sequel, also during some proofs.

  \subsection{Simple comparisons}
  \label{comp}
  As is well known, the   combinatorial rank   is the analogue of the   rank for a divisor on a smooth curve, in the following sense.
If $X$ is      smooth   and $D$ is a divisor on it 
we have
\begin{eqnarray*}
&r(X,D)=h^0(X,D)-1=\;\;\\
%&\max\{k: \forall E\in \Div^k_+X\ \  \  \exists    D'\sim D, \text{ such that }  D'-E\geq 0\}
&=\max\{k: \forall p_1,\ldots, p_k\in X\  \  \exists D'\sim D:   \    D' -p_i\geq 0\   \  \forall i=1\ldots k\}. 
 \end{eqnarray*}
 Now, if 
  $X$ is singular the above identity may fail, as the next example shows.
  First, recall that two Cartier divisors, $D$ and $D'$, on   $X$ are defined to be linearly equivalent,
  in symbols $D\sim D'$,
  if the corresponding line bundles, or invertible sheaves, $\O_X(D)$ and $\O_X(D')$,
  are isomorphic.
  \begin{example}
 Let $X=C_1\cup C_2$ be the union of two smooth rational curves meeting at a point 
(a node of $X$). Let $q\in C_2$ be a smooth point of $X$; then  $r(X,q)=1$
(see the next remark). Now, for any smooth point  $p$ of $X$ lying on $C_1$   we have  $q\not\sim p$  
(these two divisors   have different multidegree).
\end{example} 
We will use the  following 
simple  facts. 

\begin{remark}
\label{rkin}
Let
$X=Z\cup Y$ with  $Z$ and $Y$ connected subcurves with no common components,
set $k:=|Z\cap Y|$. Pick
 $L\in \Pic X$, then:
\begin{enumerate}
\item $r(Z,L_Z)+r(Y,L_Y) -k+1\leq r(X,L)\leq r(Z,L_Z)+r(Y,L_Y)+1.$
%We have
%$$\  \  r(Z,L_Z)+r(Y,L_Y) -k+1\leq r(X,L)\leq r(Z,L_Z)+r(Y,L_Y)+1,$$
\item
 If $k=1$ we have  $r(X,L)= r(Z,L_Z)+r(Y,L_Y)+1$ if and only if
 $L_Z$ and $L_Y$ have a base point at the branch over   $Z\cap Y$.
 \item
%\label{rkvan}
%If $r(Z,L_Z)=-1$ (e.g. 
If 
$\mdeg L_Z<0$ we have
$ 
 r(X,L)=  r(Y,L_Y(-Y\cdot Z)),
$  where $Y\cdot Z$ denotes the degree-$k$  divisor cut by $Z$ on $Y$.
\end{enumerate}
\end{remark}

Let $X$ be a nodal connected curve and $G$ its {\it dual graph}. Recall that $G$ is 
defined so that
the set of its vertices  is identified with the set of irreducible components of $X$ (we always use notation \eqref{Vdec}),
 the set of its edges is identified with the set of nodes of $X$, and for $v,w\in V(G)$
we have
$(v\cdot w)=|C_v\cap C_w|$.
The weight  function on $G$ assigns to the vertex $v$ the genus of the desingularization of the corresponding component, $C_v$.
The arithmetic genus of $X$ is equal to the genus of its dual graph.

 \
 
The  divisor  theory of $G$ is best connected to the divisor theory of $X$
by adding to the picture variational elements, i.e. by considering one-parameter
families of curves specializing to $X$, as follows.

Let $\phi:\X \to B$ be  a    regular one-parameter smoothing  of a curve $X$. That is,
$B$ is a smooth connected one-dimensional variety 
with a  marked point $b_0\in B$ , $\X$ is a   regular surface,
and $\phi^{-1}(b_0)\cong X$ while $\phi^{-1}(b)$ is  a smooth curve for every $b\neq b_0$.
Such a  $\phi$ determines a discrete subgroup  $\Tw_{\phi}X$ of $\Pic^0(X)$:
\begin{equation}
\Tw_{\phi}X:=\{\O_{\X}(D)_{|X},\  \forall D\in \Div (\X): \supp D\subset X\}/\cong.
\end{equation}
 Elements of $\Tw_{\phi}X$ are called  twisters.
The multidegree map 
$$
\mdeg :\Tw_{\phi}X \la \Z^{\gamma}=\Div(G)
$$
has image, independent of $\phi$, written
$$
\Lambda_X=\mdeg \ (  \Tw_{\phi}X ) \subset\Div^0(G).
$$
We now connect   with the divisor theory of $G$.
Write $X=\cup_{v_i\in V(G)}C_{v_i}$;
it is obvious that $\Lambda_X$ is generated by
 $\mdeg\ \O(C_{v_i})$ for   $i=1,\ldots,\gamma$.
On the other hand   we clearly have
$$
\mdeg\ \O(C_{v_i})=((v_1\cdot v_i),\ldots, (v_{\gamma}\cdot v_i))=-\div f_i 
$$
where $f_i:V(G)\to \Z$ is the function
taking value $+1$ at $v_i$ and zero elsewhere.
Therefore 
  $\mdeg\ \O(C_{v_i})\in \Prin (G)$. Finally, as the  set    $\{\div(f_i),\  \  i=1,\ldots\gamma\}$ generates $\Prin (G)$, we obtain
$$\Lambda_X =
 \Prin(G).$$ 
For $v\in V(G)$ we shall denote
 \begin{equation}
\label{mt}
\mt_v:=\mdeg\ \O(C_{v})= ((v_1\cdot v),\ldots, (v_{\gamma}\cdot v))\in \Prin(G).
\end{equation}
By \eqref{vv} any $\gamma-1$ elements of type $\mt_v$  generate $\Prin(G)$.

\

 We denote by $q_{\phi}: \Pic (X)\to   \Pic (X)/ \Tw_{\phi}X$ the quotient map. Summarizing, we have a commutative diagram
  \begin{equation}\label{2graphs}
\xymatrix{
  \Div (X) \ar[r] & \Pic (X) \ar[r]^{\mdeg} \ar[d]_{q_{\phi}}
&  \Div (G)   \ar[d]^{q_{G}} \\
&  \Pic (X)/ \Tw_{\phi}X \ar[r] 
%
%\frac{ \Pic (X)}{Tw_fX}
  &   \Div (G)/\Prin G =\Pic(G). \\
}
\end{equation}
We are going to use the diagram to compare the combinatorial rank $r_{G}(\md)$ to
the algebraic rank $r(X,L)$, where $L$ is a line bundle  on $X$.
The next statement summarizes a series of well known facts by highlighting opposite behaviours.
 
 \begin{prop} [Differences in combinatorial and algebraic setting] 
  \label{rdiff} 
  
  \
  
  \noindent
  Let $X$ be a reducible curve and  $G$ its dual graph. 
 \begin{enumerate}[{\bf (1)}]
  \item
  \label{rd}
  \begin{enumerate}[(a)]
\item For every $d\in \Z$ and  $\md\in  \Div ^d(G)$  we have $r_{G}(\md)\leq \max\{-1,d\}$.

 \item For every 
 $d,n\in \Z$   there exist infinitely many $\md$ with $|\md|=d$ such that
   $r(X,L)> n$ for every  $ L\in \Pic^{\md}(X)$.
 \end{enumerate}

 \item
  \label{rq}
  \begin{enumerate}[(a)]
\item 
For any   $\md, \md' \in  \Div (G)$ with   $\md\sim \md'$
  (i.e.  $q_{G}(\md)=q_{G}(\md')$)
we have $r_{G}(\md)=r_{G}(\md')$.
 \item For every  regular one-parameter smoothing $\phi$ of $X$   there exist infinitely many $L,L'\in  \Pic (X)$ with
  $q_{\phi}(L)=q_{\phi}(L')$
and  $r(X,L)\neq r(X,L')$.
 \end{enumerate}
 \item
 \label{r+}
\begin{enumerate}[(a)]
 \item  \cite[Lemma 2.1]{BNRR} 
 For any $\md, \md' \in  \Div (G)$ with $r_G(\md) \geq 0$ and $r_G(\md')\geq 0$  we have 
 $$r_{G}(\md)+r_{G}(\md')\leq r_{G}(\md +\md').$$
  \item 
  There exist    infinitely many $L,L'\in  \Pic (X)$ with $r(X,L) \geq 0$ and $r(X,L')\geq 0$ such that 
 $$r (X,L)+r (X,L')> r (X,L\otimes L').$$
 \end{enumerate}
  \item
   \label{rcl}  
\begin{enumerate}[(a)]
 \item \cite[Cor. 3.5]{BNRR}(Clifford for graphs)
For any  $0\leq d\leq 2g-2$   and  any   $\md\in  \Div^d(G)$ we have
 $$r_{G}(\md)\leq d/2.$$
\item 
For any  $0\leq d\leq 2g-2$ there exist  infinitely many $\md$ with $|\md |=d$
such that for any $L\in \Pic^{\md}(X)$
$$r(X,L)>d/2.$$
 \end{enumerate}
\end{enumerate}
\end{prop}
\begin{remark}
In \cite{BNRR} the authors work with loopless, weightless graphs, but it is clear that the two above results extend, using definition \eqref{rankwl}.
\end{remark}
\begin{proof}
Part (\ref{rd}). The assertion concerning $r_{G}$ follows immediately from the   definition.
The second part  follows from the next observation.
 
Let     $\md=(d_1,\ldots, d_{\gamma})$ be any multidegree on $X$.
For any integer $m$ we   pick  $\md '=(d_1',\ldots, d_{\gamma}')\sim  \md$ such that $d_1'\geq m$
(for example $\md'=\md - \mdeg \ \O_X((m+d_1)C_1)$). 
It is clear that for any $n\in \N$ we can choose      $m$ large enough so that for every $L'\in \Pic^{\md '}(X)$
 we have $r(X,L')\geq n$.
In particular, for every $L\in \Pic ^{\md}(X)$, any regular smoothing $\phi$ of $X$, 
 there exists $L'\in  \Pic (X)$ such that 
$q_{\phi}(L)=q_{\phi}(L')$ and  $r(X,L')\geq n. $ 
From this argument we derive item (b) for parts \eqref{rd}, \eqref{rq} and \eqref{rcl}.
 
%In part (\ref{rq}), the statement about $r_{G}$ follows immediately from \eqref{trasl}.
 
It remains to prove item (b) of part (\ref{r+}).
Fix an irreducible component   $C$ of $X$ and
set $Z=\overline{X\smallsetminus C}$.
Pick any effective  Cartier divisor $E$ on $X$ with
$ 
\supp E\subset  Z
$ 
and such that, setting $L'= \O_X(E)$, we have 
\begin{equation}
\label{LE}
r(X,L')\geq 1.
\end{equation}
Now pick $m\geq 2g_C+k$ where  $g_C$ is the arithmetic genus of $C$ and 
  $k= |C\cap Z|$. Let $\md$ be a multidegree with
  $d_C=m$ and such that
    $$
  \md_Z+\mdeg_Z\O_X(E) < 0.
  $$
In particular $\md_Z<0$, hence for every   $L\in \Pic^{\md}X$ we have 
$$
r(X,L)=r(C,L(-C\cdot Z))=m-k-g_C\geq g_C\geq 0 
$$  
(writing $C\cdot Z$ for the divisor cut on $C$ by $Z$; see Remark~\ref{rkin}).
Now   consider 
$L\otimes L'=L(E)$. We have
$ 
\mdeg_ZL(E)= \md_Z+\mdeg_Z\O_X(E) < 0
$ 
hence 
$$
r(X,L\otimes L')=r(C,L(E -C\cdot Z))=r(C,L(-C\cdot Z))=r(X,L). 
$$  
By \eqref{LE}, we have $r(X,L\otimes L')<r(X,L)+r(X,L')$ and are done.
 \end{proof}
 
We now mention, parenthetically but using the same set-up,   a different type of result on the interplay between algebraic geometry and graph theory, when families of curves are involved.  
This is the Specialization Lemma of \cite{bakersp}, concerning
a regular one-parameter smoothing 
$\phi:\X\to B$  of a curve $X$ as before
(so that $X$ is   the fiber over $b_0\in B$). This lemma states that if
$\L$ is a line bundle on the total space $\X$ then, up to shrinking $B$ near $b_0$,  for every $b\in B\smallsetminus \{b_0\}$  
  the algebro-geometric rank of the restriction of $\L$ to the fiber over $b$
is at most equal to the combinatorial rank of the multidegree
of the restriction of $\L$ to $X$. In symbols, for all $b\neq b_0$, we have
$ 
r(\phi^{-1}(b), \L_{|\phi^{-1}(b)})\leq r_{G}(\mdeg \ \L_{|X}).
$ 
 (This form is actually a   generalization of the one proved in \cite{bakersp}; see    \cite{AB} and \cite{AC}.)
 Apart from being interesting in its own right, the Specialization Lemma has some remarkable applications, like a new proof of the classical Brill-Noether theorem (see \cite{ACGH}) given in \cite{CDPR}. We view this as yet another motivation to study the algebro-geometric meaning of the combinatorial rank.
 
 \
 
A fundamental analogy between the algebraic and combinatorial setting is the   Riemann-Roch formula, which holds for every nodal curve $X$ and every graph $G$.
The algebraic case is   classical:  
let  $K_X\in \Pic (X)$ be the dualizing line bundle (equal to the canonical bundle if $X$ is smooth), then
  for any Cartier divisor $D$ on $X$ we have $$
r(X,D)-r(X,K_X(-D))=\deg D -g +1
$$
where $g$ is the arithmetic genus of $X$. 

The same formula holds for graphs.
To state it,  we  
introduce   the canonical divisor, $\mk_{G}$, of  a graph $G$:
\begin{equation}
\label{candef}
\mk_{G}:=\sum_{v\in V(G)}\bigr(2\omega(v)-2+\val(v)\bigl)v
%(2\omega(v_1)-2+\val(v_1),\ldots, 2\omega(v_{\gamma})-2+\val(v_{\gamma})) 
\end{equation}
where $\val(v)$ is the valency of $v$.
If $G$ is the dual graph of   $X$ we have 
\begin{equation}
\label{canGX}
\mk_{G}=\mdeg \  K_X.
\end{equation}
\begin{thm}
[Riemann-Roch formula for graphs]
Let  $G$ be a graph of genus $g$; for every       $\md\in \Div^d(G)$
 we have
 $$r_{G}(\md)-r_{G}(\mk_{G}-\md)=d-g+1.
$$
\end{thm}
This is \cite[Thm 1.12]{BNRR} for loopless, weightless graphs;
the extension  to general graphs can be found in \cite{AC}.

From Riemann-Roch we immediatly derive the following facts.
\begin{remark}
\label{st}
Let $\md\in \Div^0(G)$. Then $r_G(\md)\leq 0$ and equality holds if and only if $\md\sim \mo$.

Let $\md\in \Div^{2g-2}(G)$. Then $r_G(\md)\leq g-1$ and equality holds if and only if $\md\sim \mk_{G}$.

\end{remark}
\subsection{Edge contractions and smoothings of nodes}

Let    $S\subset E(G)$ be a set of edges. By $G/S$ we denote the graph obtained by contracting to a point
(i.e. a vertex of $G/S$)  every edge in $S$; 
the associated map will be  denoted by
$$\sigma:G\to G/S.$$  
There is an obvious identification  
  $E(G/S)=E(G)\smallsetminus S$. The map $\sigma$ induces a surjection
  $$
  \sigma_V:V(G)\la V(G/S);\  \  \  v\mapsto \sigma(v).
  $$
For   $\oov\in V(G/S)$ we set      $\ov{\omega}(\oov)=\sum_{v\in   \sigma_V^{-1}(\oov)}\omega(v)+b_1(\sigma^{-1}(\oov))$ for its weight,
so that $\ov{\omega}(\oov)$ is  the genus of the (weighted) graph $\sigma^{-1}(\oov)$.
  We refer to $G/S$ as a {\it contraction} of $G$; notice that $G$ and $G/S$ have the same genus.
  A picture  can be found in Example~\ref{failsc}.
\begin{remark}
\label{smoothcont}
Contractions   are particularly interesting for us, as they correspond
to  ``smoothings" of  algebraic curves. 
More precisely, let $\phi:\X\to B$ be a one-parameter family of curves having $X$ as special fiber,
and let $n\in X$ be a node; we say that $\phi$ is a smoothing of $n$ if $n$ is not the specialization
of a node of the generic fiber
(i.e. if there is an open neighborhood $U\subset \X$ of $n$ such that 
the restriction of $\phi$ to $U\smallsetminus n$ has  smooth fibers).
Let $G$ be the dual graph of $X$ and let $S\subset E(G)$ be the set of edges corresponding to nodes $n$ such that $\phi$ is a smoothing of $n$. Then,   the contraction
$G/S$   is the dual graph of the fibers of $\phi$ near $X$.
The converse also holds, i.e. for any contraction $G\to G/S$ there exists a deformation of $X$ smoothing precisely the nodes corresponding to $S$.
\end{remark}
  
Observe now that associated to $\sigma:G\to G/S$  there is   a map
$$
\sigma_*: \Div (G) \la  \Div (G/S);\quad \quad  \sum_{v\in V(G)}n_v v\mapsto 
\sum_{\ov{v}\in V(G/S)}\bigr(\sum_ {v\in \sigma_V^{-1}(\ov{v})}n_v \bigl)\ov{v}.
$$
We need the following  fact (essentially due to Baker-Norine, \cite{BNhyp}).
 \begin{prop}
 \label{s*}
 Let $G$ be a graph,  $e\in E(G)$, and let $\sigma:G\to  G/e$ be the contraction of $e$.
 Then
  \begin{enumerate}
  \item
  \label{contfun}
$\sigma_*: \Div (G) \to  \Div (G/e)$ 
is a surjective  group homomorphism such that
 $\sigma_*(\Prin (G))\supset \Prin (G/e).$ 
 \item
 \label{bridge}
 $\Pic (G) \cong \Pic (G/e)$ if and only if $e$ is a bridge
 (i.e. a separating edge).
 In this case the above isomorphism is induced by $\sigma_*$, and $\sigma_*$ preserves the rank.
 \end{enumerate}
\end{prop} 
\begin{proof}
It is clear that $\sigma_*$ is a surjective homomorphism.
Let $v_0,v_1\in V(G)$ be the endpoints of $e$.
Set  $\ov{G}:=G/e$, now write    
$V(G)=\{v_0, v_1,  \ldots, v_n\}$
and $V(\ov{G})=\{\ov{v_1}, \ldots, \ov{v_n}\}$ with $\sigma_V(v_i)=\ov{v_i}$ for $i\geq 1$.

Denote by $\mt_i=((v_0\cdot v_i), (v_1\cdot v_i),\ldots,(v_n\cdot v_i))\in \Prin (G)$   the principal divisor corresponding to $v_i$, defined in \eqref{mt},
and  by $\ov{\mt_i}$   the principal divisor of ${\ov{G}}$ corresponding to $\ov{v_i}$.
As we mentioned earlier, it suffices to show that  $\ov{\mt_i}\in \sigma_*(\Lambda_{G})$ for $i=2,\ldots,n$.
This follows from the    identity 
\begin{equation}
\label{push}
\sigma_*(\mt_i)=\ov{\mt_i},\  \quad  \forall i=2,\ldots,n.
\end{equation}
Let us  prove it for $i=2$ (which is obviously enough).
We have 
$$
\sigma_*(\mt_2)=((v_0\cdot v_2)+(v_1,\cdot v_2),  (v_2 \cdot v_2),\ldots,(v_n \cdot v_2)),
$$
now $(v_0\cdot v_2)+(v_1,\cdot v_2)=({\ov{v_1}}\cdot {\ov{v_2}})$  and  
 $(v_i \cdot v_2)=({\ov{v_i}}\cdot {\ov{v_2}})$ for every $i\geq 2$
hence (\ref{push}) is proved.

Part \eqref{bridge}. Suppose  $e$ is a bridge;
then by  \cite[Lm. 5.7, Cor. 5.10]{BNhyp} there is a rank-preserving isomorphism
$\Pic(G^{\bullet})\cong \Pic(G^{\bullet}/e)$. Of course, $G^{\bullet}/e=(G/e)^{\bullet}$, hence
by \eqref{inj.},  we obtain a rank preserving isomorphism $\Pic(G)\cong \Pic(G/e)$.

Assume $e$ is not a bridge. 
Recall that for any $d $ and any  $G$  the set $\Pic^dG$ has
  cardinality equal to the complexity, $c(G)$, of $G$.
Therefore it is enough to prove that $G$ and ${\ov{G}}$
have different complexity.
Now, it is easy to see that the contraction map $\sigma:G\to \ov{G}$ induces a bijection
between the spanning trees of $\ov{G}$ and the spanning trees of $G$ containing $e$.
On the other hand, since $e$ is not a bridge,   $G$ admits a spanning tree not containing $e$
(just pick a spanning tree of the connected graph $G-e$).
We thus proved that $c(G)>c(\ov{G})$, and we are done.
\end{proof}

We observed in Remark~\ref{smoothcont} that
 one-parameter families of curves correspond to
  edge contractions of graphs.
Now, in algebraic geometry the rank of a divisor is an upper-semicontinuous function:
given a family of curves $X_t$ specializing to a curve $X$, with a family of divisors $D_t\in  \Div (X_t)$
specializing to $D\in  \Div (X)$, we have $r(X_t,D_t)\leq r(X,D)$.
 
 Do we have a corresponding   semicontinuity for the combinatorial rank?
  The answer in general is no.
 By Proposition~\ref{s*}, contraction of bridges preserves the rank. But
the following example illustrates    that the rank can 
both decrease  or increase
if a non-bridge is contracted.

\begin{example}
\label{failsc}{\it Failure of semicontinuity under edge contractions.}
Consider the contraction of the edge $e_4\in E(G)$ for the graph $G$ in the picture below.
\begin{figure}[!htp]
$$\xymatrix@=1pc{
&*{\circ} \ar@{-}[rrr] _{e_2}^(0.1){v_1}^(1.1){v_2}\ar@{-}@/^.5pc/[rrr]^{e_1} \ar@{-}[ddrrr]_{e_3} ^(1.1){v_3} && &*{\circ}
\ar@{-}[dd]^{e_4}
& & 
 &\\
G \, =  &&&
&&\ar[rrrr] &&& &&*{\circ}\ar@{-}[rrr]_{e_2} ^(1.1){w_2}\ar@{-}@/^.5pc/[rrr]^{e_1} ^(1.1){w_1}
\ar@{-}@/_1.2pc/[rrr]_{e_3} &&&*{\circ} & \, = G/{e_4} \\
&& &&*{\circ} \& &
&&
&&&&
}$$
\caption{Contraction of $e_4$}
\end{figure}
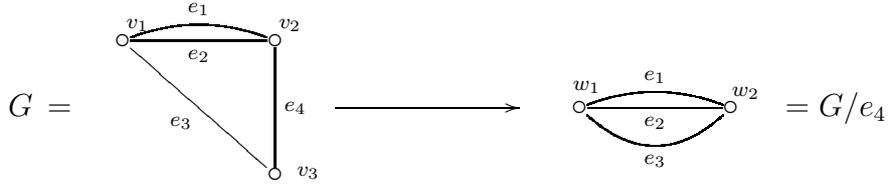

Let us first show that  the combinatorial rank may decrease.
Pick $\md=(-2,3,-1)\in \Div(G)$; then
 $r_{G}(\md)=0$ as 
 $$
 \md=-\mt_{v_2}\sim  (0,0,0).
 $$ Now
$\sigma_*(\md)=(-2,2)$ and hence  
$$r_{G/e_{4}}(\sigma_*(\md))=-1<r_G(\md).
$$

Now let us show that the combinatorial rank may   go up.
Consider $\md=(1,-1,1)\in \Div(G)$; then one checks easily (or by Lemma~\ref{kZ})
that $r_{G}(\md)=-1.$
Now
$\sigma_*(\md)=(1,0)$ hence  $r_{G/e_{4}}(1,0)=0>r_G(\md)$.
 
Let us give also an example with  $r_{G} \geq 0.$
Pick  $\me=(1,-1,2)$ so that 
$$
r_{G/e_{4}}(\sigma_*(\me))=r_{G/e_{4}}(1,1)=1.
$$
Now $\me+\mt_{v_3}=(1,-1,2)+(1,1,-2)= (2,0,0)$, hence 
  $r_{G}(\me)\geq 0.$ To show that $r_{G}(\me)\leq 0$
we note that if we subtract $(0,0,1)$ from $\me$ we get $(1,-1,1)$, which has rank $-1$,
 as  observed above. 

\end{example}
A convenient computational tool is provided by the following Lemma,
of which we  had originally a slightly less general version; the following version was suggested by the referee.

\begin{lemma}
\label{kZ}
Fix an integer  $r\geq 0$ and let $\md \in \Div(G)$ be such that for some $v\in V(G)$ we have 
$\md(v)<r$. Assume that
for every subset of vertices $Z\subset V(G)\smallsetminus \{v\}$
we have $\md(Z)< (Z\cdot Z^c)$. Then $r_{G}(\md)\leq r-1.$
\end{lemma}
\begin{proof}
Since both hypotheses remain valid in $\lm$, and 
$r_{G}(\md)$
 is defined as the rank of $\md$ on $\lm$,
  we can assume $G$ weightless and loopless.

For notational consistency, 
write $\me\in \Div^1_+(G)$ for the (effective) divisor corresponding to $v$.
By contradiction, suppose $r_{G}(\md)\geq r$; hence  $r_{G}(\md-r\me)\geq 0$, but $\md-r\me$ is not effective by hypothesis.
Therefore 
for some  nontrivial principal divisor  $\mt=\div(f)\in \Prin(G)$ we have
$$
0\leq  \md-r\me+\mt. 
$$   
%By \cite[Lm. 3.4]{Cner}  there exists $Z\subset V(G)$ such that 
We use Remark~\ref{cnerlm}; let $Z\subset V(G)$ be the set of vertices where
$f$ assumes its minimum; then 
 $\mt(Z)\leq -(Z\cdot Z^c)$. We have $v\not\in Z$,
for otherwise $\mt(v)\leq 0$ hence $ (\md-r\me+\mt)(v)<r-r=0$
which is impossible.
Therefore, by hypothesis,   $\md(Z)< (Z\cdot Z^c)$, which  yields
(as $\me(Z)= 0$)
$$
 0\leq (\md-r\me +\mt)(Z)=  \md(Z)-r\me(Z)+ \mt(Z)\leq  \md(Z)-(Z\cdot Z^c)<0,
$$
  a contradiction.
\end{proof}

\section{Algebraic interpretation of the combinatorial rank}
\label{sec2}

Let $G$ be a graph of genus least 2. We say $G$     is  {\it semistable}
if  every vertex of weight zero has valency at least 2, and we say $G$ is 
{\it stable} if every vertex of weight zero has valency at least 3.
This terminology is motivated by the fact that  a curve $X$ of arithmetic    genus at least 2 is semistable, or stable, 
  if and only if so is its dual graph.

\subsection{A conjecture}
If $G$ is a stable graph, 
the locus of isomorphism classes of curves whose dual graph is $G$
is an interesting subset of the moduli space of stable curves, denoted 
 $\MaG\subset \Mgb$;
 it is well known that $\MaG$ is irreducible, quasiprojective of dimension $3g-3-|E(G)|$.
More generally,  i.e. for any graph, we denote by
 $\MaG$   the set of isomorphism classes of curves having $G$ as dual graph.

Let $X\in \MaG$ and $\md \in \Div(G)$,
we denote
$$
%r^{\rm{max}}_X^{\md}
\RX:=\max \{r(X,L),\  \    \forall L\in \Pic^{\md}(X)\}.%=\max \{r: W^r_{\md}(X)\neq \emptyset\}.
$$
By Riemann-Roch  we have
\begin{equation}
\label{RRin}
\RX\geq \max\{-1, |\md|-g\}.
\end{equation}
We want to study the relation between $r_G(\md)$ and $\RX$.
Now, the combinatorial rank $r_G$ is constant in an equivalence class, hence
we set, for any $\delta\in \Pic (G)$ and $\md\in \delta$
$$
r_{G}(\delta):=r_G(\md).
$$
On the other hand, we saw in Proposition~\ref{rdiff} that the algebraic rank behaves badly
with respect to linear equivalence of multidegrees, indeed, it is unbounded on the fibers of $q_{\phi}$. Therefore
  we set
$$
\rdel:=\min \{\RX,\  \  \  \forall \md\in \delta\}.
$$
Now,  having the   analogy with \eqref{toy} in mind,
we state
\begin{conj}
\label{toy2}
Let $G$ be a graph  and $\delta\in \Pic ^d(G)$.
Then
$$
 r_{G}(\delta)=\max \{r(X,\delta),\  \   \forall X\in \MaG\}.
$$
\end{conj}
We set
$$
\rgdel:=\max \{\rdel,\  \  \  \forall X\in \MaG\},
$$
so that the above conjecture becomes
\begin{equation}
\label{thalg}
\rgdel=r_{G}(\delta).
\end{equation}
We think of $\rgdel$ as the ``algebro-geometric" rank of the combinatorial class $\delta$.
We shall prove that \eqref{thalg} holds in low genus and for $d\geq 2g-2$.

\begin{remark}
Stable and semistable curves are of fundamental importance in algebraic geometry; see \cite{gac}, \cite{DM}, \cite{HMo}.
We shall see, as a consequence of Lemma~\ref{tail}, that if Identity   \eqref{thalg} holds for semistable graphs, it holds for any graph.
\end{remark}
The following is a simple evidence for the conjecture. 
 
\begin{lemma}
\label{zero}
 Conjecture~\ref{toy2}  holds for
 $\delta =0$. More precisely for every $G$ and $X\in \MaG$ we have 
 $\RX=r_{G}(\delta) =0$.
\end{lemma} 
  \begin{proof}
We have $r_G(\delta)=0$, of course.
 Now, as we explained in Subsection~\ref{comp}, every $\md\in \delta$ is the multidegree of some twister of $X$; pick one of them, $T$, so that
 $T\in \Pic^{\md}(X)\cap \Tw_{\phi}(X)$ for some regular one-parameter smoothing $\phi$.
 By upper-semicontinuity of the algebraic rank, the twister $T$, being the specialization of the trivial line bundle,
 satisfies  $r(X,T)\geq 0$.
On the other hand $r(X,\O_X)=0$ and  it is easy to check that any other $L\in \Pic^{\mo}(X)$ has rank $-1$;
%(by \cite[Cor. 2.2.5]{Ctheta});
so we are done.
\end{proof}

Here is an example where  Conjecture~\ref{toy2}   holds, and the equality
$\rdel=r_G(\delta)$ does not hold for every $X\in \MaG$.

\begin{example}
\label{binex}
Let $G$ be a {\it binary graph} of genus $g\geq 2$, i.e. $G$ is the graph  with two vertices of weight zero joined by $g+1$ edges. (This graph is sometimes named ``banana" graph; we prefer the word binary for consistency with the terminology used in other papers, such as \cite{Ccliff}.) 
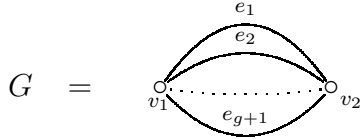
\begin{figure}[h]
\begin{equation*}
\xymatrix@=.5pc{
 &\\
G&=  &&*{\circ}
  \ar @{.} @/_.2pc/[rrrrr]_(0.01){v_1} \ar @{-} @/_1.5pc/[rrrrr]^{e_{g+1}} _(1){v_2}\ar@{-} @/^1.1pc/[rrrrr]^{e_2}\ar@{-} @/^2pc/[rrrrr]^{e_1}
&&&&& *{\circ} &&&&
\\
 &\\
}
\end{equation*}
\end{figure}
 
 Let $\md=(1,1)\in \Div(G)$.
It is clear that $r_G(\md)=1$.

Let now $X$ be a curve whose dual graph is $G$, 
so $X$ has two smooth rational components intersecting in $g+1$ points; we say
 $X$ is  a binary curve.
It is easy to check   that Clifford's theorem holds in this case (i.e. for this multidegree), hence $r(X,L)\leq 1$ for every $L\in \Pic^{(1,1)}(X)$.

Suppose first that $g=2$. Then we claim that for every such $X$ we have $\RX=1$
and there exists a unique   $L\in \Pic^{(1,1)}(X)$ for which $r(X,L)=1$.
Indeed, to prove the existence it suffices to pick  $L=K_X$. The fact that
  there are no other line bundles with this multidegree and rank follows from Riemann-Roch.
  
  Now let $g\geq 2$. 
We say that a binary curve $X=C_1\cup C_2$ is   {\it special} if there is an isomorphism of pointed curves
 $$(C_1;p_1,\ldots p_{g+1})\cong (C_2;q_1,\ldots q_{g+1})$$
 where $p_i,q_i$ are the branches of the $i$-th node of $X$, for $i=1,\ldots g+1$ 
 (if $g=2$ every binary curve is special).
 
 We claim that
 $\RX=1$ if and only if $X$ is   special, 
 and in this case there exists a unique  $L\in \Pic^{(1,1)}(X)$ for which $r(X,L)=1$.
 We use induction on $g$; 
 the base case $g=2$ has already been done .
Set $g\geq 3$ and observe that the desingularization of a special binary curve  at a node is again special.
 
Let $\nu_1:X_1\to X$ be the desingularization of $X$ at one node, 
so that $X_1$ has genus $g-1$. Let $p ,q \in X_1$ be the branches of the  desingularized node.
By induction $X_1$ admits a line bundle $L_1$ of bidegree $(1,1)$ and rank $1$ if and only if $X_1$ is special, and in this case $L_1$ is unique.
Next,  there exists $L \in \Pic ^{(1,1)}(X)$ having   rank $1$ if and only if $X_1$ is special,
$\nu_1^*L=L_1$ and,  
$$
r(X_1,L_1(-p))=r(X_1,L_1(-q))=r(X_1,L_1(-p-q))=0;
$$
moreover such $L$ is unique if it exists
(see  \cite[Lm. 1.4]{Ccliff}).
% (see also  \cite[Lm. 2.4, 2.5]{Ctheta}).
Therefore $L_1=\O(p+q)$, hence $X$ is a special curve. The claim is proved.

 Let us now consider  $\md'\sim \md$ with $\md'\neq \md$: 
 $$
 \md'=(1+n(g+1),1-n(g+1)).
$$
By symmetry  we can assume $n\geq 1$.
Then for any $L\in \Pic^{\md'}X$ we have
$$
r(X,L)=r(C_1, L_{C_1}(-C_1\cdot C_2))=r(\PP^1, \O((n-1)g+n))=(n-1)g+n\geq 1.
$$
Therefore, denoting by $\delta\in \Pic(G)$ the class of $\md=(1,1)$ we have
 $\rdel =\RX
$ for every $X\in \MaG$.

Here is a summary of what we   proved.

\noindent
{\it Let $G$ be a binary graph of genus $g\geq 2$,  $\md=(1,1)$
and $\delta\in \Pic(G)$ the class of $\md$. Pick $X \in \MaG$, 
then }
 \begin{displaymath}
\rdel=\RX=\left\{ \begin{array}{ll}
1  \  &\text{ if $X$ is special} \\
0 \  &\text{ otherwise.}\\
\end{array}\right.
\end{displaymath}
{\it And if $X$ is special there exists a unique $L\in \Pic^{(1,1)}(X) $ having rank $1$.}
\end{example}

\subsection{Low genus cases.}

We use the following terminology. A vertex $v\in V(G)$ of weight zero  and  valency one
is a {\it leaf-vertex}, and the edge $e\in E(G)$  adjacent to $v$ is a {\it leaf-edge}.
Note that a leaf-edge is a bridge.

 Let $\sigma:G\to \ov{G}=G/e$ be the contraction of a leaf-edge.
 By Proposition~\ref{s*} 
 the map $\sigma_*:\Div (G)\to \Div(\ooG)$
induces an isomorphism 
 $$
 \sigma_*:\Pic(G)\stackrel{\cong}{\la} \Pic (\ooG)
 $$
(abusing notation).  Let $X\in \MaG$, then the component $C_v$ corresponding to the leaf-vertex $v$ is a smooth rational curve attached   at a unique node;
such components are called {\it rational tails}.
Now, we have a natural surjection
$$
\MaG \la M^{\operatorname{alg}}(\ov{G});\quad\quad X\mapsto \ov{X}
$$
where  $\ooX$ is  obtained from $X$ by removing  $C_v$.
Here is a picture,  useful also for    Lemma~\ref{tail}.
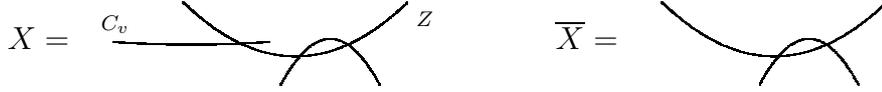
 \begin{figure}[h]
\begin{equation*}
\xymatrix@=.5pc{
&&&\ar@{-}@/_2pc/[rrrrrrr]_>{Z}&&&&&&&&&&&&\ar@{-}@/_2pc/[rrrrrrr]&&&&&&&&&\\
X=& \ar@{-} @/_.1pc/[rrrrr]^<{C_v} &&&&&&&&&&&&&\ov{X}=&
&&&&&&&&\\
&&&&&&\ar@{-}@/^1.6pc/[rrr] &&
&&&&&&&&&&\ar@{-}@/^1.6pc/[rrr]&&&\\}
\end{equation*}
\end{figure}
\begin{lemma}
\label{tail}
Let $G$ be a graph and $\sigma:G\to \ov{G}=G/e$ the contraction of a leaf-edge.
For every $\delta\in \Pic(G)$ and every $X\in \MaG$ we have, with the above notation,
$$
r(X,\delta)=r(\ooX, \sigma_*(\delta)).
$$
In particular, Identity   \eqref{thalg} holds for $G$ if and only if it holds for $\ooG$.
\end{lemma}
\begin{proof}
Let $v\in V(G)$ be the leaf-vertex of $e$ and  $C=C_v\subset X$ the corresponding rational tail; we write  $X=C\cup Z$ with $Z\cong \ooX$, and   identify $Z=\ooX$ from now on.
Pick $\md\in \delta$ and set $c=\md(v)$;
we define
$$
\md^0:=\md +c\mt_v
$$
where $\mt_v\in \Prin(G)$ 
was  defined in \eqref{mt}. Hence    $\md^0(v)=0$ and $\md^0\sim \md$. Notice that $\sigma_*(\md)=\sigma_*(\md^0)$. Now, 
since  $C\cap Z$ is a separating node of $X$, 
there is a canonical isomorphism $\Pic X\cong \Pic (C)\times \Pic (Z)$
mapping $L$ to 
the pair of its restrictions, $(L_C,L_Z)$.
Hence 
% as $\Pic^0(C)$ is a point,
we have an isomorphism
$$
\Pic^{\md^0}(X)\stackrel{\cong}{\la} \Pic^{\sigma_*(\md^0)}(\ooX);
%= \Pic^{\sigma_*(\md^0)}(Z);
\quad\quad\quad L\mapsto \ov{L}:=L_{Z},
$$
 as for any $L\in \Pic^{\md^0}(X)$ we have $L_{C}=\O_C$. Moreover, we have
$$
r(X,L)=r(Z, L_{Z})=r(\ooX, \ov{L}) 
$$
 by Remark~\ref{rkin}. Therefore
\begin{equation}
\label{rmax}
\rmax(X,\md^0)=\rmax(\ooX, \sigma_*(\md^0)).
\end{equation}
Now we claim that for every $\md\in \delta$ we have
\begin{equation}
\label{clr}
\rmax(X,\md)\geq \rmax(X,\md^0).
\end{equation}
This claim implies our statement. In fact
it implies that $r(X,\delta)$ can be computed by looking only at representatives taking value $0$ on $C$,
i.e.
$$
r(X,\delta)=\min\{\rmax(X, \md^0), \  \forall \md^0\in \delta\};
$$
now by \eqref{rmax} and the fact that $\sigma_*:\Div(X)\to \Div(\ooX)$ is onto we get
 $$
r(X,\delta)=\min\{\rmax(\ooX, \ov{\md}), \  \forall  \ov{\md}\in \sigma_*(\delta)\}=r(\ooX,  \sigma_*(\delta)) 
$$
and we are done. 

We now prove \eqref{clr}.
By what we said before,
line bundles on $X$ can be written as pairs
$(L_C,L_Z)$.
Pick $L\in \Pic^{\md}(X)$ and set $L^0:=(\O_C, L_Z(cp))$
where $p=C\cap Z\in Z$ and $c=\deg_CL$ as before.
Hence $L^0\in \Pic^{\md^0}(X)$ and this sets up a bijection
$$
\Pic^{\md}(X)\la \Pic^{\md^0}(X);\quad \quad \quad L\mapsto L^0.
$$
We shall prove $r(X,L)\geq r(X,L^0)$ for every $L\in \Pic^{\md}(X)$, which clearly implies \eqref{clr}.
If $c\geq 0$ we have
$$
r(X,L)\geq r(C,\O(c))+r(Z,L_Z)=c+r(Z,L_Z)
$$
and
$$
r(X,L^0)=r(Z,L_Z(cp))\leq c+r(Z,L_Z); 
$$
combining the two inequalities we are done. If $c<0$ we have
$$
r(X,L)=r(Z,L_Z(-p))\geq r(Z,L_Z(-|c|p))=r(X,L^0).
$$
The proof is finished.
\end{proof}
Let $G$ have  genus $g\geq 2$ and let $\ooG$ be obtained after  all possible leaf-edges contractions; then $\ooG$ is a semistable graph.
By the previous result  we can assume all graphs and curves of genus $\geq 2$ semistable.

\begin{cor}
\label{cor0}
 Conjecture~\ref{toy2}  holds if $g=0$.
\end{cor}
\begin{proof}
By Lemma~\ref{tail} we can assume $G$ has one  vertex (of weight zero) and no edges, so that
the only curve in $\MaG$ is $\PP^1$.
Now every $\delta\in \Pic^d(G)$,
has a unique representative  and
$r_G(\delta)=\max\{-1,d\}$. On the other hand  
$\Pic^d(\PP^1)=\{\O(d)\}$  and
$r(\PP^1, \O(d))=\max\{-1,d\}$.  \end{proof}
Another consequence of Lemma~\ref{tail} is the following.
\begin{prop}
\label{cor1}
 Conjecture~\ref{toy2}  holds if $g=1$.
\end{prop}
\begin{proof}
By Riemann-Roch we have, for every $\delta\in \Pic^d(G)$
 \begin{displaymath}
r_G(\delta)=\ \left\{ \begin{array}{ll}
d-1 \  &\text{ if } d\geq 1 \\
\  \  0 \  &\text{ if } \delta=0\\
-1  & \text{ otherwise. }\\
\end{array}\right.
\end{displaymath}
By Lemma~\ref{tail} we can assume $G$ has no leaves.
If $G$ consists of a vertex of weight $1$ 
then a curve  $X\in \MaG$ is   smooth of genus $1$, and the result follows from Riemann-Roch.

So we can assume $G$ is a cycle with $\gamma$ vertices, all 2-valent
of weight zero, and $\gamma$ edges.
Now, we   have $|\Pic^d(G)|=\gamma$ (as the complexity of $G$ is obviously $\gamma$).
Let us exhibit the elements of $\Pic^d(G)$ by suitable representatives:
$$
\Pic^d(G)=\{[(d, \mo_{\gamma-1})],[(d-1, 1,\mo_{\gamma-2})],\ldots, [(d-1, \mo_{\gamma-2},1)]\}
$$
where we write $\mo_i=(0,\ldots, 0)\in \Z^i$.
We need to show  the above $\gamma$ multidegrees are not equivalent to one another;
indeed the difference of any two of them is of type $\pm(\mo_i,1,\mo_j,-1,\mo_k)$
which has rank $-1$ (by Lemma~\ref{kZ} for example).

Pick  now $X\in \MaG$.
Assume $d\geq 1$.
By Riemann-Roch $r(X,L)\geq d-1$ for any line bundle $L$ of degree $d$, so it suffices to show that 
every $\delta\in \Pic^d(G)$ has a representative $\md$ such that for some $L\in \Pic^{\md}(X)$
equality holds.
Let $\md$ be any of the above representatives and
pick $L\in \Pic^{\md}(X).$ It is   easy to check directly that $r(X,L)=d-1$
(or,  one can apply \cite[Lm. 2.5]{Ccliff}), 
so we are done.

Suppose $d\leq 0$; by Lemma~\ref{zero} we can assume $\delta\neq 0$.
Let $\md$ again be any of the  above  representatives.
One easily see that $r(X,L)=-1$ for every $L\in \Pic^{\md}(X)$
(as a nonzero section of $\O_{\PP^1}(1)$ cannot have two zeroes).
Hence $\rdel =-1=r_G(\delta)$ for every $X\in \MaG$. The result is proved.
  \end{proof}

The proof of the next proposition contains some computations that could be avoided using later
results. Nevertheless we shall  give the direct proof, which   explicitly illustrates
   previous  and later topics.
\begin{prop}
\label{g2thm}
 Conjecture~\ref{toy2}  holds for stable graphs of genus $2$.
\end{prop}
\begin{proof}
Let $G$ be a stable graph of genus $2$ and $\delta\in \Pic^d(G)$.
In some cases $r_G(\delta)$ is independent of $G$; namely
if $d<0$ then $r_G(\delta)=-1$, and if $d\geq 3$ then $r_G(\delta)=d-2$
 by
  \cite[Thm 3.6]{AC}. For the remaining cases we need to know $G$.
As  $G$ is stable, it has at most two vertices;
the case $|V(G)|=1$ is treated just  as for higher genus, so we postpone it
to Corollary~\ref{corir}.
If $|V(G)|=2$
there are only two possibilities,
  which we shall treat separately. 
  We shall use   Remark~\ref{rkin} several times without mentioning it.

\

\noindent {\bf Case 1.}
 $G$ has only one edge   and 
both vertices of  weight 1.
Below we have a picture of $G$  together with its weightless model $G^{\bullet}$,
and  with a useful  contraction   of $G^{\bullet}$:
 \begin{figure}[h]
\begin{equation*}
\xymatrix@=.5pc{
G = &*{\bullet}\ar@{-}[rr]_<{+1}_>{+1}
&&*{\bullet}&&&
G^{\bullet} = &*{\circ}\ar@{-} @/^.5pc/[rr]\ar@{-} @/_.5pc/[rr] &&*{\circ}\ar@{-}[rr]^{e}
&&*{\circ}\ar@{-} @/^.5pc/[rr]\ar@{-} @/_.5pc/[rr] &&*{\circ}
&&&
G^{\bullet}/e = &*{\circ}\ar@{-} @/^.5pc/[rr]\ar@{-} @/_.5pc/[rr] &&*{\circ}\ar@{-} @/^.5pc/[rr]\ar@{-} @/_.5pc/[rr] &&*{\circ}
}
\end{equation*}
\end{figure}
 
Clearly, we can identify
 $\Pic(G)= \Z$. Next
 denoting by $e$ the bridge of $\G^{\bullet}$,  by Proposition~\ref{s*} we have a rank preserving
isomorphism
$$
\Pic(G^{\bullet})\cong \Pic(G^{\bullet}/e).
$$
 %and  $\Pic^0(G^{\bullet})\cong   \Z/2\Z\oplus \Z/2\Z.$ 
 Finally, since  there is an injection $\Pic(G)\ha \Pic(G^{\bullet})$ we also have
$$
\Pic(G)\ha   \Pic(G^{\bullet}/e); \quad \quad [(d_1,d_2)]\mapsto [(0, d_1+d_2, 0)]
$$
where we ordered the vertices  from left to right  using the picture.

For any $X\in \MaG$,
we have $X=Z\cup Y$ with $Z$ and $Y$ smooth of genus 1, intersecting at one point. 

 If $d<0$ we pick  the representative
 $(0,d)\in \delta$. Then $\rmax(X,(0,d))=-1$,  hence $r(X,\delta)=-1$
and we are done. 
If  $d\geq 3$ 
 we pick $(d_1,d_2)\in \delta$ with $d_1\geq 1$ and $d_2\geq 2$ so that
  $$
  \rmax(X,(d_1,d_2))=d_1-1+d_2-1=d-2=r_G(\delta);
$$
by \eqref{RRin} we are done.
The case $\delta=0$ is   in \ref{zero}.
The remaining two cases, $d=1,2$ are done in the second and third column of the table below.
The combinatorial rank is computed on $G^{\bullet}/e$.
 For the algebraic computations we used  also the symmetry of the situation.
The two consecutive rows starting with $r_G(\md)$ and $\rmax(X,\md)$
prove that $r(X,\delta) \leq r_G(\delta)$; the last  row shows that equality holds.

\begin{center}
\begin{tabular}{|c||c|c|}
\hline
&& \\
 $[\md]\in \Pic(G)$        & $[(0,1)]$   & $[(0,2)]$\\
&& \\
\hline
%&& \\
 $[\md^{\bullet}]\in \Pic(G^{\bullet}/e)$       & $[(0,1,0)]$   & $[(0,2,0)]$\\
%&& \\
\hline
&& \\
 $r_G(\md)=$      & $0$  & $1$\\
&& \\
\hline
&& \\
 $\rmax(X,\md)=$       & $0$  & $1$\\
&   & \\
\hline
$\md'\sim \md $    & $(a,1-a)$  & $(a,2-a)$\\
\hline
& &\\
 $\rmax(X,\md')=$       & $\left\{ \begin{array}{ll}
a-1\geq 1 & \text{   } a\geq 2\\
 -a \geq 1 &\text{   } a\leq -1\\
\end{array}\right.$
 & $\left\{ \begin{array}{ll}
a-1\geq 2 &\text{   } a\geq 3\\
1 &\text{   } a=1\\
1-a\geq 2 &\text{   } a\leq -1\\
\end{array}\right.$\\
&  &\\
\hline
\end{tabular}
\end{center}

Case 1 is finished.  

\

\noindent {\bf Case 2.} $G$ is a binary graph, as in Example~\ref{binex}, with 3 edges.
We have $\Pic^0(G)\cong \Z/3\Z$.
If $d<0$ or $d\geq 3$ we know $r_G(\delta)$; for the remaining cases
we listed
 the rank of each class in the table below, 
with a choice of representatives making the computations  trivial (by Lemma~\ref{kZ}). \begin{center}
\begin{tabular}{|c||c|c|c|}
\hline
&&&\\
 $d=0$        & $r_G(0,0)=0$   & $r_G(1,-1)=-1$& $r_G(2,-2)=-1$\\
 &&&\\
 \hline
 &&&\\
 $d=1$        & $r_G(0,1)=0$   & $r_G(1,0)=0$& $r_G(2,-1)=-1$\\
 &&&\\
 \hline
  &&&\\
 $d=2$        & $r_G(0,2)=0$   & $r_G(1,1)=1$& $r_G(2,0)=0$\\
 &&&\\
 \hline
\end{tabular}
\end{center}
Let now $X\in \MaG$; we already described such curves in Example~\ref{binex},
where we proved the result for   $\delta=[(1,1)]$,  which we can thus skip, as well as   $\delta=[(0,0)]$.
We   follow   the rows of the table.
If $d=0$ and $a=1,2$ we have for any $L\in \Pic^{(a,-a)}(X)$,
\begin{equation}
\label{0-1}
r(X,L)=r(\PP^1,\O(a-3))=-1=r_G(a,-a).
\end{equation}
The case $d=0$ is done. Next, 
$\rmax(X,(0,1))\leq 0$, and it is clear if $L=\O(p)$, with $p$ nonsingular point of $X$,
we have $r(X,L)=0$; hence 
 $\rmax(X,(0,1))=0.
$ 
For the other multidegrees in $[(0,1)]$ we have
$$
r(X,(3a,1-3a))=\left\{ \begin{array}{ll}
r(\PP^1,\O(3a-3)=3a-3\geq 0 &\text{ if } a\geq 1\\
r(\PP^1,\O(-3a-2)=-3a-2\geq 1 &\text{ if } a\leq -1.\\
\end{array}\right.$$
So $r(X,[(0,1)])=0=r_G([(0,1)]).$
As for the last class of degree 1, 
for every $X$ and $L\in \Pic^{(2,-1)}(X)$ we have
$$
r(X,L)=r(\PP^{1}, \O(-1))=-1=r_G(2,-1)
$$
hence this case is done.

We are left with $\delta=[(0,2)]$; we claim  $r(X,\delta)=0$ for every $X$.
By Riemann-Roch $r(X,L)\geq 0$ for any $L\in \Pic^2(X)$, so we need to prove
that for some $\md\in \delta$ equality holds for every $L\in \Pic^{\md}(X)$;
choose $\md=(3,-1)$, then
$ 
r(X,L)=r(\PP^1, \O(3-3))=0
$ 
as claimed.

To finish the proof notice that
 $r(X,\delta)=-1$ if $d<0$ (easily done arguing as for \eqref{0-1}). Finally, we claim
  $r(X,\delta)=d-2$ if $d\geq 3$. For this we pick for $\delta$ a representative
$(d_1,d_2)$ with $d_1\geq 0$ and $d_2\geq 3$; then one checks easily
that $\rmax(X,(d_1,d_2))=d-2$; by \eqref{RRin} we are done.
\end{proof}

\subsection{High degree divisors and irreducible curves}
Recall that we
  can assume all graphs and curves semistable of genus at least $2$.
The following theorem states that
 if  $d\geq  2g-2$  then   Identity   \eqref{thalg}  is true  in a   stronger form.
 First  we need the following.
\begin{defi}
Let $G$ be a semistable graph of genus $g\geq 2$, and let $\md\in \Div^dG$.
We say that $\md$ is  
  {\em{semibalanced}}   if  
for every   $Z\subset V(G)$ the following 
inequality holds
\begin{equation}
\label{BIeq}
 %\frac{d\mk_G(Z)}{2g-2}-\frac{Z\cdot Z^c}{2}\leq \md(Z)  
\md(Z) \geq   \mk_G(Z)d/(2g-2)-(Z\cdot Z^c)/2 
\end{equation}
 and if for every vertex $v$ of weight zero and valency 2  we have $\md(v)\geq 0$.

We say that $\md$ is {\em{balanced}} 
 if it is semibalanced and if for every  vertex $v$ of weight zero and valency 2    we have $\md(v)= 1$.
 \end{defi}
 The reason for introducing  this  technical definition (the graph theoretic analogue of \cite[Def. 4.6]{Cner})
is that for line bundles of semibalanced multidegree
 we have extensions of Riemann's, and partially Clifford's, theorem, as we shall see in the proof of the next theorem.

\begin{thm}
\label{Riemann}
Let $G$ be a semistable graph of genus $g$ and assume $d\geq 2g-2$. Then for every
$\delta\in \Pic^d(G)$  the following facts hold.
\begin{enumerate}
\item
\label{Rie1}
 Conjecture~\ref{toy2}  holds.
\item
\label{Rie2}
There exists $\md\in \delta$ such that  $\rmax(X,\md)=r_G(\md)$ for
  every $X\in \MaG$.
\item
\label{Rie3} Every semibalanced $\md\in \delta $
satisfies part \eqref{Rie2}.
  \end{enumerate}
\end{thm}

 \begin{proof}
We have that  
every $\delta\in \Pic G$ admits a semibalanced representative
(see \cite[Prop. 4.12]{Cner}). Therefore  \eqref{Rie3} implies \eqref{Rie2}, which   obviously implies \eqref{Rie1}.  We shall now prove  \eqref{Rie3}.
 
If $d\geq 2g-1$, by \cite[Thm 3.6]{AC} we have $r_G(\delta)=d-g$.

On the other hand, by the Riemann-Roch theorem for curves, we have  $r(X,L)\geq d-g$ for every line bundle $L$ of degree $d$. 

Now, by the extension of Riemann's theorem to singular curves \cite[Thm 2.3]{Ccliff}, for every  balanced  representative $\md\in \delta$,
and  for every $L\in \Pic^{\md}$, we have  
\begin{equation}
\label{riembal}
r(X,L)=d-g.
\end{equation}
Hence if $\md$ is balanced we are done.
  It remains to show that the theorem we just used extends to semibalanced multidegrees. A balanced multidegree $\md$ is defined as a semibalanced one,
satisfying the extra condition $\md(v)=1$ for any vertex $v$ of weight zero and valency $2$.
Now it is simple to check that the proof of that theorem  
never uses the extra condition,
hence \eqref{riembal} holds also for any $L$ of semibalanced multidegree.
 This completes the proof in case $d\geq 2g-1$.
 
 Now assume $d=2g-2$. By Remark~\ref{st} we have $r_G(\delta)\leq g-1$ with equality if and only if $\delta$ is the canonical class. Let $\md\in \delta$ be semibalanced.
 By \cite[Thm 4.4]{Ccliff} (an extension of Clifford's theorem), if $\md$ is such that for every subcurve  $Z\subsetneq X$
 of arithmetic genus $g_Z$ 
 we have the following inequality 
 \begin{equation}
\label{44}\md(Z)\geq 2g_{Z}-1, 
\end{equation} 
then we have 
$ \rmax(X,\md)\leq g-1$
with equality if and only if $\md=\mdeg K_X$;
 as
 $\mdeg K_X=\mk_G$  we will be done  if  \eqref{44}  holds for every
subcurve  $Z$. 

To prove that, we abuse notation writing $Z\subset V(G)$ for
the set of vertices  corresponding to the components of $Z$.
 As $\md$ is semibalanced  we have  
 $$
 \md(Z)\geq \mk_G(Z)-(Z\cdot Z^c)/2=2g_{Z}-2+(Z\cdot Z^c)-(Z\cdot Z^c)/2
 $$
as by  \eqref{canGX} we have
 $\mk_G(Z)=\deg_{Z}K_X=2g_{Z}-2+(Z\cdot Z^c)$.
 Therefore 
 $$
 \md(Z)\geq 2g_{Z}-2+(Z\cdot Z^c)/2\geq  2g_{Z}-3/2,
 $$
(as $(Z\cdot Z^c)\geq 1$) which implies $\md(Z)\geq 2g_{Z}-1$. So \eqref{44} holds and we are done.
\end{proof}

\begin{cor}
 Conjecture~\ref{toy2}  holds if $d\leq 0$. 
 
 To prove Conjecture~\ref{toy2} in all remaining cases it suffices to prove it for  $d\leq  g-1$.
\end{cor}
\begin{proof}
For $\md\in \Div(G)$ set
$ 
\md^*=\mk_G-\md 
$ so that $|\md^*|=2g-2-d$.
Then, by Riemann Roch,
$\rmax(X,\md)=r_G(\md)$ if and only if $\rmax(X,\md^*)=r_G(\md^*)$.
Therefore 
the Conjecture holds for $[\md]$ if and ony if it holds for $[\md^*]$.

If $d\leq 0$ then $|\md^*|\geq 2g-2$ and the Conjecture holds by
Theorem~\ref{Riemann}. If 
$d\geq g$ then $|\md^*|\leq g-2$, so we reduced to   the required range.
\end{proof}

\begin{cor}
\label{corir}
 Conjecture~\ref{toy2}  holds if $|V(G)|=1$, i.e. if $\MaG$ parametrizes irreducible curves.
\end{cor}
\begin{proof}
The graph $G$ consists of a vertex $v$ of weight $h$ and $g-h$ loops
attached to $v$, with $0\leq h\leq g$; recall that we can assume $g\geq 2$.
Let $\delta =[d]\in \Pic G$;   we can assume $1\leq d\leq g-1$.
 By \cite[Lemma 3.7]{AC} we have
 $ 
 r_G(d) =  \left\lfloor{\frac{d}{2}}\right\rfloor.
 $ 

Let now $X\in \MaG$; as $X$ is irreducible   
Clifford's theorem  holds, hence
$r(X,L)\leq \left\lfloor{\frac{d}{2}}\right\rfloor$ for every $L\in \Pic^d(X)$.
We must prove there exists $X\in \MaG$ admitting  $L\in \Pic^d(X)$ for which equality holds.
If $d=1$ we take $L=\O_X(p)$ with  $p$   nonsingular point of $X$; then $r(X, \O_X(p))=0$. We are left with the case  $g\geq 3$;
 it is well known that
  $\MaG$ contains a hyperelliptic curve, $X$, and that
 %by  \cite[Lm 5.2.3]{Ctheta}, 
there exists
   $L\in \Pic^d(X)$ for which $r(X,L)=\left\lfloor{\frac{d}{2}}\right\rfloor$. So we are done.
\end{proof}
For convenience, we collect together all the cases treated in the paper.
\begin{summ}
Let $G$ be a  (finite, connectected, weighted) graph of genus $g$ and let $\delta\in \Pic^d(G)$. Then Conjecture~\ref{toy2} holds  in the following cases.
\begin{enumerate}
\item $g\leq 1$.
\item $d\leq 0$ and $d\geq 2g-2$.
\item
$|V(G)|=1.$
\item
$G$ is a stable graph of genus 2.
\end{enumerate}
\end{summ}

\end{document}